\documentclass [12pt,reqno]{amsart}

\usepackage{graphicx} 
\usepackage{pifont}
\usepackage{graphicx}
\usepackage{xcolor}
\usepackage{amsmath}
\usepackage{amsthm}
\usepackage{amsfonts}
\usepackage{amssymb}
\usepackage{url}
\usepackage{setspace}  
\usepackage{lipsum}

\newtheorem{theorem}{Theorem}[section]

\theoremstyle{definition}

\theoremstyle{remark}
\newtheorem{remark}[theorem]{Remark}

\theoremstyle{lemma}

\theoremstyle{proposition}
\newtheorem{proposition}[theorem]{Proposition}

\begin{document}
\title[\tiny{Weak $(1,1)$ bounded operators}]{Weak $(1,1)$ bounded operators}
	
	\author{Arup Kumar Maity}

\address{ Department of Mathematics, NIT Sikkim, Ravangla-737139, India}

\email{arup.anit@gmail.com}
	\begin{abstract}
		We construct a class of Fourier multipliers whose associated operators are weak $(1,1)$ bounded but fail to be weak $(p,p)$ bounded for any $1<p \leq \infty$. Moreover, we show that this result is sharp.
	\end{abstract}
	
	
	
	\maketitle
	
\section{Introduction}
\noindent Let $1 \leq p < \infty$, and denote by  $L^p(\mathbb{R}^n)$ the space of $p$-integrable functions on $\mathbb{R}^n$. The space $L^{\infty}(\mathbb{R}^n)$ consists of essentially bounded measurable functions. In the context of distribution theory, we also consider the Schwartz space $\mathcal{S}(\mathbb{R}^n)$ consists of all smooth functions $f$ satisfying the rapid decay condition $$|x^{\alpha}\partial^{\beta}f(x)| < B(f, \alpha, \beta)$$ for every pair of multi-indices $\alpha, \beta$. The topology of $\mathcal{S}(\mathbb{R}^n)$ is defined by a countable family of seminorms $$p_N(f)=\text{Sup}_{|\alpha|, |\beta| \leq N}\|x^{\alpha}\partial^{\beta}f(x)\|_{\infty},$$ for each non-negative integer $N$. The dual space of $\mathcal{S}(\mathbb{R}^n)$, denoted by $\mathcal{S}'(\mathbb{R}^n)$, consists of tempered distributions, which are continuous linear functionals of $\mathcal{S}(\mathbb{R}^n)$ (see \cite{GF}). Given a tempered distribution $m$ in $\mathcal{S}'(\mathbb{R}^n)$, define the operator $T_m : \mathcal{S}(\mathbb{R}^n) \to \mathcal{S}'(\mathbb{R}^n)$ by $$\widehat{T_mf}= m\hat{f},$$ where $\hat{f}$ is the Fourier transform of $f$. The right-hand side is a tempered distribution defined by $$\langle m\hat{f}, \phi \rangle= \langle m, \hat{f}\phi \rangle,$$ where $\phi \in \mathcal{S}(\mathbb{R}^n)$. It follows that the of $T_m$ on $f$ can be expressed as the convolution $$ T_mf = \Check{m} * f,$$ where $\Check{m}$ is the inverse Fourier transform of $m$. The Fourier transform extends to an isomorphism on the space of tempered distributions $\mathcal{S}'(\mathbb{R}^n)$, thus making the above formulation well-defined. In this context, $m$ is referred to as a Fourier multiplier, and the associated operator $T_m$ is called a Fourier multiplier operator. For further discussion on Fourier multipliers and their properties, we refer the reader to \cite{B, H1, H2, hor, si}. 

The note introduces a special class of multipliers $m$ that behave just well enough to be weak $(1,1)$ bounded, but fail to be weak
$(p,p)$ bounded for any 
$1 < p \leq \infty$. An operator $S: L^p(\mathbb{R}^n) \to L^{p, \infty}(\mathbb{R}^n)$ is weak $(p,p)$ bounded if for a positive $\alpha$ we have $$|\{x: |Sf(x)|> \alpha\}| \leq C(\|f\|_p/\alpha)^p,$$ where $C$ is independent of $f$. For $1\leq q \leq 2$, consider the class $$\mathcal{A}=\left \{\sum_if_i*g_i: \sum_i\|f_i\|_1\|g_i\|_q < \infty\right \}.$$ That kind of class was studied in \cite{F}, to classify the multiplier space. First note that $\mathcal{A} \subset L^q(\mathbb{R}^n)$. One can check that $\mathcal{A}$ is a Banach space with norm $$|||h|||_{\mathcal{A}}= \text{inf}_{f_i, g_i} \left \{\sum_i\|f_i\|_1\|g_i\|_q : h=\sum_if_i*g_i\right \},$$ where $f_i, g_i$ varies as one $\sum_if_i*g_i$ may have a 
different representation $\sum_i\Tilde{f_i}*\Tilde{g_i}$. We show that for some $q$ in that range, the elements of $\mathcal{A}$ (thought of as Fourier multipliers $m$) give rise to operators $T_m$ that are weak $(1,1)$ but not weak $(p,p)$ for the other $p$. 

Finally, we show for those $m \in \mathcal{A}$ the operator $T_m$ is $L^1(\mathbb{R}^n) \to L^{1,r}(\mathbb{R}^n)$ not bounded for $1\leq r < \infty$. So $T_m$ is weak $(1,1)$ bounded, which is sharp in the above sense. Here $L^{1,r}(\mathbb{R}^n)$ is the usual Lorentz space with norm $$\|f\|_{1,r}= \left(\int_0^{\infty}(tf^*(t))^r\frac{dt}{t}\right)^{\frac{1}{r}},$$ where $f^*$ is the non-increasing rearrangement of $f$ (see \cite{Z2}). Most familiar Fourier multipliers (e.g., Hilbert transform, Riesz transforms) are bounded on all $L^p(\mathbb{R}^n)$, $1< p < \infty$, and weak $(1,1)$ at the endpoint. Here, we have a very delicate example: bounded only at $p=1$ in the weakest possible sense, failing for any $p>1$.
This shows that certain factorizable $L^q(\mathbb{R}^n)$ symbols are "just barely" good enough for $L^1$-endpoint control.
\section{Weak $(1,1)$ bound}
\noindent In the current section, we prove our main result. Before that, we need to state another well-known result (stated in \cite{Z1}), a useful and important sufficient condition for boundedness in singular integral theory. It is useful to prove our principal theorem. 
\begin{theorem} \label{SI}
    Let $S: L^1(\mathbb{R}^n) \to M(\mathbb{R}^n)$ be a linear operator which is defined by $$Sf(x)=K*f(x),$$ where $K$ is the kernel and $M(\mathbb{R}^n)$, space of measurable functions. Then the operator $S$ is weak $(1,1)$ bounded if $$\int_{S^{n-1}}|K(x)|\text{log}(1+|K(x)|)d\sigma_{x} < \infty,$$ where $d\sigma_x$ is the surface measure of the unit sphere.
\end{theorem}
Now, we state the Stein-Tomas restriction theorem (\cite{Tao}).
\begin{theorem}
    Let $1 \leq q \leq \frac{2n+2}{n+3}$. Then, for $f \in \mathcal{S}(\mathbb{R}^n)$ we have $$\|\widehat{f}\|_{L^2(S^{n-1})} \leq \|f\|_q$$
\end{theorem}
In the next theorem, we show that the operators whose symbol is an element of $\mathcal{A}$, are weak $(1,1)$ bounded.
\begin{theorem}
    Let $m \in \mathcal{A},$ and $ 1 \leq q \leq \frac{2n+2}{n+3}$. Then the corresponding multiplier operator $T_m$  is weak $(1,1)$ bounded.
\end{theorem}
\begin{proof}
    Given that $m=\sum_if_i*g_i$. As $m \in L^q(\mathbb{R}^n)$ and $L^q(\mathbb{R}^n) \subset L^1(\mathbb{R}^n)+L^2(\mathbb{R}^n)$, the Fourier transform of $m$ is meaningful here. So $\Check{m}=\sum_i\Check{f_i}\Check{g_i}$. Also, $$T_mf(x)=(\sum_i\Check{f_i}\Check{g_i})*f(x).$$ Now, we use the Theorem \ref{SI}. We have to estimate the following term
        $$\int_{S^{n-1}}|\sum_i\Check{f_i}\Check{g_i}(x)|\text{log}(1+|\sum_j\Check{f_j}\Check{g_j}(x)|)d\sigma_{x}.$$
  By Fubini's theorem, this term is bounded by      $$\Sigma_i\int_{S^{n-1}}|\Check{f_i}\Check{g_i}(x)|\text{log}(1+|\Sigma_j\Check{f_j}\Check{g_j}(x)|)d\sigma_{x},$$  which is dominated by
$$\Sigma_i\|\Check{f_i}\|_{\infty}\int_{S^{n-1}}|\Check{g_i}(x)|\text{log}(1+|\Sigma_j\Check{f_j}\Check{g_j}(x)|)d\sigma_{x}.$$ Now, using the inequality $\text{log}(1+|x|) \leq |x|,$ we have the above term is bounded by
$$\Sigma_i\|\Check{f_i}\|_{\infty}\int_{S^{n-1}}|\Check{g_i}(x)||\Sigma_j\Check{f_j}\Check{g_j}(x)|d\sigma_{x},$$ is dominated by
$$\Sigma_{i,j}\|\Check{f_i}\Check{f_j}\|_{\infty}\int_{S^{n-1}}|\Check{g_i}(x)||\Check{g_j}(x)|d\sigma_{x}.$$ After that, by H\"older's inequality on $S^{n-1}$ we get
$$\Sigma_{i,j}\|\Check{f_i}\Check{f_j}\|_{\infty}\|\Check{g_i}(x)\|_{L^2(\sigma_x)}\|\Check{g_j}(x)\|_{L^2(\sigma_x)}.$$ Finally, using restriction and Hausdorff-Young's inequality we have
$$\Sigma_{i,j}\|{f_i}\|_{1}\|{f_j}\|_{1}\|{g_i}\|_{q}\|{g_j}\|_{q},$$ which is equal to
$$ (\Sigma_{i}\|{f_i}\|_{1}\|{g_i}\|_{q})^2.$$ The last term is finite as $m \in \mathcal{A}, 1 \leq q \leq \frac{2n+2}{n+3}$. So by the above theorem, the operator $T_m$ is weak $(1,1)$ bounded.
\end{proof}
\begin{remark}
    The operator $T_m$ defined in the preceding theorem is not weak $(p,p)$ bounded for any $1<p\leq \infty$. Indeed, if $T_m$ were weak $(p,p)$ bounded for some $p\neq 1$ then by the Marcinkiewicz interpolation theorem it would be strong $(s,s)$ bounded for all $1<s<p$. In particular, strong $(2,2)$ boundedness would imply that $m \in L^{\infty}(\mathbb{R}^n)$, which is not the case in general. This contradiction shows that $T_m$ cannot be weak $(p,p)$-bounded for any $p>1$. \end{remark}
    \begin{remark}
        The multipliers in the class $\mathcal{A}$ do not satisfy the smoothness and decay conditions of H\"ormander or Mikhlin type, which guarantee weak $(1,1)$ boundedness. Consequently, the standard multiplier theorems do not apply, and the proof method developed here is essential to handle these cases. 
    \end{remark}
    \section{Sharpness}
    \noindent In this section, we show the sharpness of the main result. We are motivated to do that from \cite{gra}.
    \begin{proposition}
    Let $m \in \mathcal{A}$. Then the operator $T_m: L^1(\mathbb{R}^n) \to L^{1,r}(\mathbb{R}^n)$ is not bounded for $1\leq r < \infty$.
    \end{proposition}
    \begin{proof}
        We first treat the one dimensional case and assume that $m$ has only a single non zero term in its defining series. Let $$m=f*g,$$ where $f(y)$ and $g(y)$ are the inverse Fourier transform of $e^{-y^2}$ and $\frac{1}{y^2}\chi_{(1, \infty)}(y)$ respectively, where $\chi_{(1, \infty)}$ is characteristic function of the interval $(1, \infty)$. Consider the sequence of Schwartz functions $\{h_n\}$ defined by $$h_n(y)=\frac{1}{\sqrt{n}}e^{-\frac{x^2}{n}}.$$ It is straightforward to check that $$\|h_n\|_1=\sqrt{\pi}$$ for all $n$. We aim to estimate the Lorentz norm $$\|\hat{f}\hat{g}*h_n\|_{1,r},$$ and show that it diverges to $\infty$ as $n \to \infty.$
        By definition it is equal to $$(\int_0^{\infty}(t(\hat{f}\hat{g}*h_n)^*(t))^r\frac{dt}{t})^{\frac{1}{r}}.$$ Substituting the expressions for $f, g$ and $h_n$ and applying the inequality $$(x-y)^2 \leq 2(x^2+y^2)$$ we obtain a lower bound of the form $$\frac{C}{\sqrt{n}}\int_0^{\infty}t^{r-1}e^{-\frac{rt^2}{2n}}dt,$$ for some constant $C$ independent of $n$.  Evaluating the integral yields $$C'n^{\frac{r-1}{2}},$$ for some constant $C'$ only depend on $r$. For $r>1$ the expression tends to $\infty$ as $n \to \infty$.
    \end{proof}
\section*{Declaration}
\noindent Ethical approval: We hereby declare that this work has no conflict of interest, neither personal nor financial.\\
Availability of data and material : The data that support the findings of this study are previous research works given in the reference and cited in this work.

\end{document}